\documentclass[preprint,12pt]{elsarticle}

\usepackage{amsfonts,amssymb,amsmath,amsthm,bm}
\usepackage{mathtools}
\usepackage{ dsfont }
\usepackage{color}
\newcommand{\Var} {\mbox{$\rm{Var}$\,}}

\newtheorem{lemma}{Lemma}
\newtheorem{theorem}{Theorem}
\newtheorem{corollary}{Corollary}

\newtheorem{definition}{Definition}
\newtheorem{proposition}{Proposition}

\usepackage{lineno}

\journal{Theoretical Population Biology}

\begin{document}
%\linenumbers

\begin{frontmatter}



\title{Coalescence and sampling distributions for Feller diffusions}

\author[label1]{Conrad J.\ Burden}
\ead{conrad.burden@anu.edu.au}
\author[label2]{Robert C.\ Griffiths}
\ead{Bob.Griffiths@Monash.edu}
\address[label1]{Mathematical Sciences Institute, Australian National University, Canberra, Australia}
\address[label2]{School of Mathematics, Monash University, Australia}

\begin{abstract}
%\begin{linenumbers}
Consider the diffusion process defined by the forward equation $u_t(t, x) = \tfrac{1}{2}\{x u(t, x)\}_{xx} - \alpha \{x u(t, x)\}_{x}$ for $t, x \ge 0$ and 
$-\infty < \alpha < \infty$, with an initial condition $u(0, x) = \delta(x - x_0)$.  This equation was introduced and solved by Feller to model the 
growth of a population of independently reproducing individuals.  We explore important coalescent processes related to Feller's solution.  
For any $\alpha$ and $x_0 > 0$ we calculate the distribution of the random variable $A_n(s; t)$, defined as the finite number of ancestors at a time 
$s$ in the past of a sample of size $n$ taken from the infinite population of a Feller diffusion at a time $t$ since since its initiation.  
In a subcritical diffusion we find the distribution of population and sample coalescent trees from time $t$ back, conditional on non-extinction 
as $t \to \infty$.  In a supercritical diffusion we construct a coalescent tree which has a single founder and derive the distribution of coalescent times.
%\end{linenumbers}
\end{abstract}

\end{frontmatter}

%
%%%%%%%%%%%% Section: Introduction %%%%%%%%%%%%
%
\section{Introduction}
\label{sec:Introduction}

\citet[Section~5]{Feller1939, feller1951diffusion} introduced the process governed by the diffusion equation
\begin{equation}	\label{FellersPDE}
u_t(t, x) = \tfrac{1}{2}\{xu(t, x)\}_{xx} - \alpha\{xu(t, x)\}_{x}, \qquad 0 \le x < \infty, \; \alpha \in \mathbb{R},  
\end{equation}
as the continuum analogue of a classical branching process, primarily with the intention of modelling the growth of a population of statistically 
independently reproducing individuals.  This stochastic process has subsequently found broader applications in biology, genetics, ecology, nuclear physics, statistical physics, 
seismology and finance (see \citet[Section I]{gan2015singular} for a detailed list of references).  

In the present paper we are interested in the Feller diffusion as the continuum limit of a Bienaym\'e-Galton-Watson (BGW) process, and applications to population genetics 
including  sampling distributions of populations undergoing neutral mutations between multiple types.  Our approach is based on the coalescent.  
In common with the diffusion limit of the Wright-Fisher and Moran coalescent trees, the coalescent of a Feller diffusion has the property that it `comes down from 
infinity'.  That is to say, the entire, uncountably infinite, population at any time chosen as the present has a finite number of ancestral lineages at any positive time in the past (see Fig.~\ref{fig:FellerSolution}).  
The key step in our analysis is Theorem~\ref{thm:PolyaAeppli} in which we calculate the distribution of the discrete random variable $A_\infty(s; t)$, defined 
to be the number of ancestors at an earlier time $t - s$ of the entire population at the current time $t$ since initiation of the process.  

%%%%%%%%%%%% Figure: coming down from infinity  %%%%%%%%%%%%
\begin{figure}[t]
\begin{center}
\centerline{\includegraphics[width=1\textwidth]{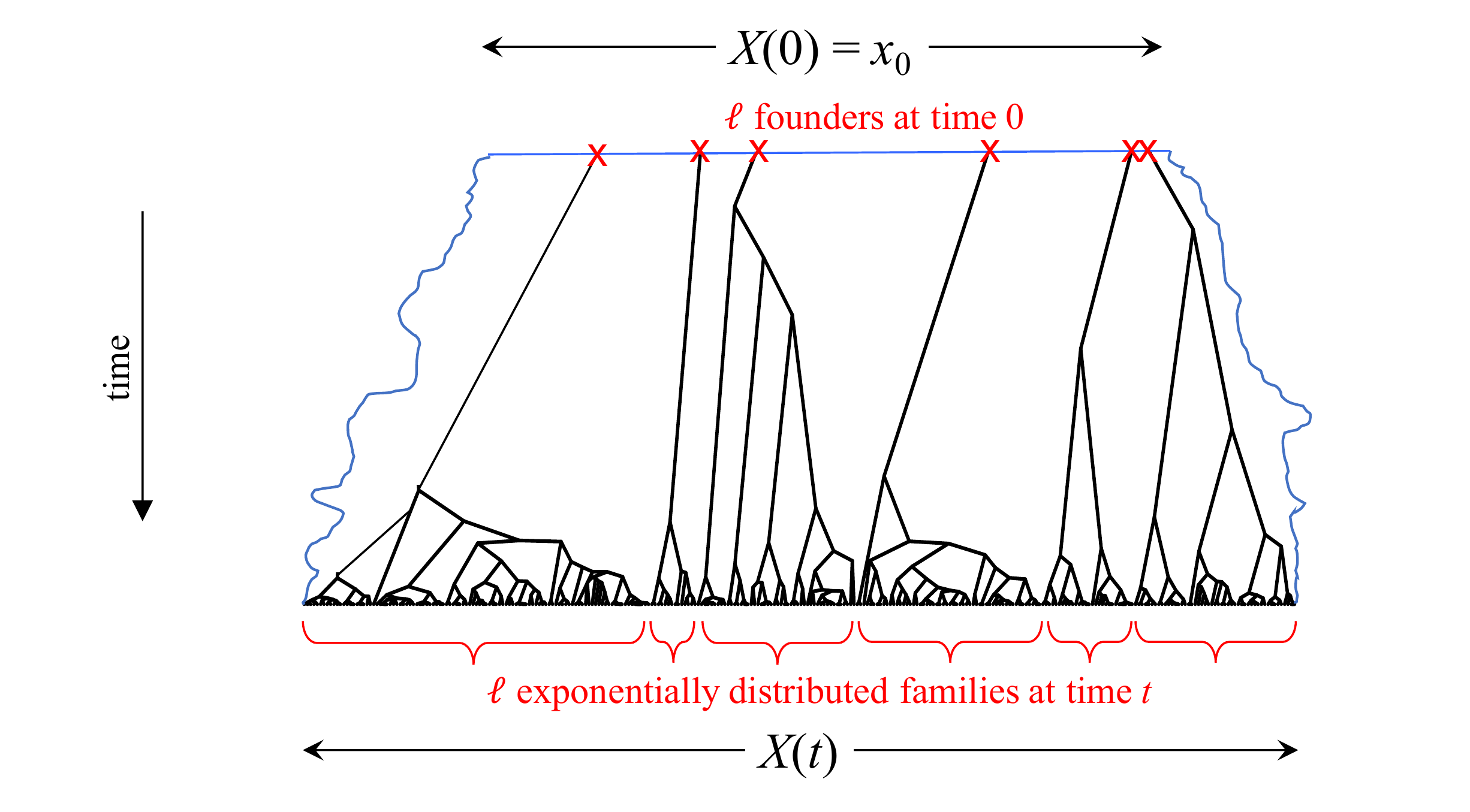}} % Plot constructed using <RRP_tree v4.R> 
\caption{Structure of the ancestral tree of a Feller diffusion conditioned on $X(0) = x_0$.  The ancestry of the population $X(t)$  
``comes down from infinity", that is, the number of ancestors (black lines) of the currently observed infinite population at any finite time in the past is finite.  
The blue lines either side represent the entire, infinite population at earlier times, most of whose descendant lines become extinct by time $t$.  The 
final population is the sum of a Poisson number of independent exponentially distributed family sizes (see Eq.~(\ref{FellersPoissonGammaMixture})).  } 
\label{fig:FellerSolution}
\end{center}
\end{figure}

Starting from this theorem we follow two paths.  In the first we consider the limit $t \to \infty$ at fixed $s$
of a subcritical Feller diffusion conditioned on survival of the population.  Results for expected inter-coalescent waiting times for a random sample agree 
with a recent, less straightforward derivation using the Laplace transform by \citet{BurdenGriffiths23}.  

The second path involves taking the limit $s \to 0$ at fixed $t$.  At any $s > 0$ we find that the ancestry of the process over the interval $[0, t - s]$ is equivalent to 
the ancestry of
a homogeneous birth-death (BD) process with $s$-dependent birth and death rates.  Both of these rates become infinite as 
$s \to 0$
while their difference remains constant 
and equal to the parameter $\alpha$.  
Useful tools for characterising the properties of coalescent trees generated by a homogeneous birth-death (BD) process are the reconstructed process 
(RP)~\citep{Nee94} and the reversed-reconstructed process (RRP)~\citep{Aldous05,Gernhard08,Stadler09,Wiuf18,Ignatieva20}, which assumes an 
improper prior on the time since initiation of the process, starting with a single ancestor.  
We construct the coalescent tree for a Feller diffusion as the limit of a RP from a linear BD process, and from this construct a RRP 
for the supercritical process and hence calculate the distribution of coalescent times.

Previously \citet{burden2016genetic} and \citet{BurdenSoewongsoso19} have studied coalescence of the supercritical Feller diffusion from a frequentist point of view, 
in the sense that they calculate confidence intervals on the time since initiation and initial parameters of the Feller diffusion in terms of a later observation.  The 
current paper is complementary in that the analysis of the supercritical diffusion is Bayesian in that the RRP requires a uniform prior on the time since initiation 
of the ancestral tree.  

The Feller diffusion is implicit within the near-critical continuum limit of a continuous-time BGW process studied by \citet{o1995genealogy} and \cite{Harris20}, who also 
calculate distributions of coalescent times.  However their analysis differs in that it assumes a process initiated by a single progenitor, their near-critical limit consists of 
taking the time since initiation to infinity while keeping the branching time and reproduction rate constant, and results are stated in terms of the fraction of time since 
initiation.  These differences make direct comparison with the current work problematic, though within the text of this paper we will note connections with \citet{o1995genealogy}
where appropriate.  

\citet{Crespo21} analyse coalescence in the infinite population limit of a RRP constructed from a BD process.  Their limiting process differs from the $s \to 0$ diffusion 
limit in the current paper in that the time axis is rescaled.  Numerical simulations of their 
coalescent model employ further approximations which amount to setting the population size as a deterministic function rather a stochastic process.  

The layout of the paper is as follows.  In Section~\ref{sec:PropertiesOfFeller} we summarise properties of the Feller diffusion including its connection 
with continuum limits of BD and BGW processes and the interpretation of Feller's solution.  This establishes known results and a notation required for subsequent sections.   
In Section~\ref{sec:Coalescence} distributions are derived for $A_n(s,\; t)$, defined to be the number of ancestors at an earlier time $t - s$ of the entire population 
at the current time $t$ since initiation of the process, and for the corresponding count $A_\infty(s; t)$ of the number of ancestors of the entire population at time $t$. 
The limit $t \to \infty$ at fixed $s$
of a sub-critical Feller diffusion conditioned on survival and the quasi-stationary sampling distribution are dealt with in Section~\ref{sec:CoalescenceTtoInfty}.  
The limit $s \to 0$ at fixed $t$ and
the connection between the RP of a BD process and the coalescent tree of a Feller diffusion are dealt with in Section~\ref{sec:sToZero}.  
In Section~\ref{sec:RRProcess} the coalescent tree for the entire population generated by a a supercritical Feller diffusion is constructed as a RRP, 
from which the distributions of coalescent times for the population and for a finite sample are calculated in Section~\ref{sec:RRPCoalescentTimes}.   
Results are summarised in Section~\ref{sec:Conclusions}.  

%
%%%%%%%%%%%% Section: Properties of the Feller Diffusion %%%%%%%%%%%%
%
\section{Properties of the Feller Diffusion}
\label{sec:PropertiesOfFeller}

For any bounded continuous function $g$ with second derivatives existing, a standard backward Kolmogorov equation for a stochastic process 
$X(t)$ is
\[
\frac{d}{d t}\mathbb{E}\left[g(X(t))\right] = \mathbb{E}\left[{\cal L}g(X(t))\right],
%\label{query:01}
\]
where the right side is the expectation of the function $h$ defined by $h = {\cal L}g$.  The operator $\cal L$ is referred to as the {\em generator} of the process.  
In general, if a continuous random variable $X(t)$ evolves so that 
\begin{equation}	\label{aAndbLimits}
\begin{split}
\mathbb{E}\left[X(t + \delta t) - X(t) \mid X(t) = x \right] & =  a(x) \delta t + o(\delta t), \\
\mathbb{E}\left[(X(t + \delta t) - X(t))^2 \mid X(t) = x\right] & =  b(x) \delta t + o(\delta t), \\
\mathbb{E}\left[(X(t + \delta t) - X(t))^k \mid X(t) = x\right] & =  o(\delta t), \quad k \ge 3,\\
\end{split}
\end{equation}
as $\delta t \to 0$, then the generator takes the form~\citetext{\citealp[p214]{KarlinTaylor81}; \citealp[Chapter~4]{Ewens:2004kx}}
\[
{\cal L} =  \tfrac{1}{2}b(x) \frac{\partial^2}{\partial x^2} + a(x) \frac{\partial}{\partial x}.
\]

For the particular case $a(x) = \alpha x$, $b(x) = x$ we refer to the process with generator  
\begin{equation}		\label{fellerGenerator}
{\cal L} = \tfrac{1}{2} x \frac{\partial^2}{\partial x^2} +  \alpha x \frac{\partial}{\partial x}, 
\end{equation}
as a {\em Feller diffusion}~\citep{feller1951diffusion,feller1951two}.   The process is said to be {\em sub-critical}, {\em critical} or {\em super-critical} if 
$\alpha <0$, $\alpha = 0$ or $\alpha > 0$ respectively.  A Feller diffusion can arise as a limit of a BD process or as 
a limit of a BGW process.  
%
%%%%%%%%%%%% Subsection: Feller Diffusion as the limit of a BD process %%%%%%%%%%%%
%
\subsection{Feller Diffusion as the limit of a BD process}
\label{sec:BPlimit}

\citet{Baake15} argue that \citet{feller1951diffusion} regarded Eq.~(\ref{FellersPDE}) primarily in terms of the large-population, near-critical 
limit of a continuous-time BD process.  In Proposition~\ref{proposition1} we set out details of how this limit scales, using notation which will 
prove useful in Section~\ref{sec:sToZero}.  

Consider a linear BD process with birth rate $\hat\lambda(\epsilon)$ and death rate $\hat \mu(\epsilon)$ specified as functions of a parameter 
$\epsilon \in \mathbb{R}_{\ge 0}$ with the property 
\[
\begin{split}
\hat\lambda(\epsilon) &= \tfrac{1}{2} \epsilon^{-1} + \tfrac{1}{2} \alpha + \mathcal{O}(\epsilon), \\
\hat \mu(\epsilon)        &= \tfrac{1}{2}\epsilon^{-1} - \tfrac{1}{2} \alpha + \mathcal{O}(\epsilon), 
\end{split}
\]
as $\epsilon \to 0$.  If 
$M_\epsilon(t)$
is the number of particles alive at time $t$, then for $n, y = 0, 1, 2, \ldots$, 
\begin{eqnarray}	\label{linearBDtransitions}
p_{n,y} &:= & \mathbb{P}\left(M_\epsilon(t + \delta t) = y \mid M_\epsilon(t) = n \right)  \nonumber \\
	& = & \begin{cases}
		n\hat\lambda \delta t + \mathcal{O}(\delta t^2), & y = n + 1; \\
		1 - n(\hat\lambda + \hat \mu)\delta t + \mathcal{O}(\delta t^2), & y = n; \\
		n \hat\mu \delta t + \mathcal{O}(\delta t^2), & y = n - 1; \\
		\mathcal{O}(\delta t^2), & \mbox{otherwise.}
	\end{cases}
\end{eqnarray}
%
%%	Proposition 1
%
\begin{proposition}\label{proposition1}
Set 
$X(t) = \epsilon M_\epsilon(t)$.
  Then the limiting generator of the process $X(t)$ as $\epsilon \to 0$ is the Feller diffusion generator Eq.~(\ref{fellerGenerator}).  
\end{proposition}
\begin{proof}
With $x = \epsilon n$ consider the simultaneous limit  $\epsilon \to 0$, $n \to \infty$ at fixed $x$:
\begin{eqnarray*}
\lefteqn{\lim_{\epsilon \to 0, n \to \infty} \mathbb{E}\left[X(t + \delta t) - X(t) \mid X(t) = x \right]  }\\
	& = & \lim_{\epsilon \to 0, n \to \infty} \epsilon  (p_{n, n+1} \times 1 + p_{n, n} \times 0 + p_{n, n - 1} \times (-1)) \\
	& = & \lim_{\epsilon \to 0, n \to \infty} \epsilon \{n (\alpha + \mathcal{O}(\epsilon))\delta t + \mathcal{O}(\delta t^2)\} \\
	& = & \alpha x \delta t + \mathcal{O}(\delta t^2);  
\end{eqnarray*}
\begin{eqnarray*}
\lefteqn{\lim_{\epsilon \to 0, n \to \infty} \mathbb{E}\left[(X(t + \delta t) - X(t))^2 \mid X(t) = x \right]  }\\
	& = & \lim_{\epsilon \to 0, n \to \infty} \epsilon^2  (p_{n, n+1} \times 1 + p_{n, n} \times 0 + p_{n, n - 1} \times 1) \\
	& = & \lim_{\epsilon \to 0, n \to \infty} \epsilon^2 \{n (\epsilon^{-1} + \mathcal{O}(\epsilon))\delta t + \mathcal{O}(\delta t^2)\} \\
	& = &  x \delta t + \mathcal{O}(\delta t^2);  
\end{eqnarray*}

and for $k \ge 3$
\begin{eqnarray*}
\lefteqn{\lim_{\epsilon \to 0, n \to \infty} \mathbb{E}\left[(X(t + \delta t) - X(t))^k \mid X(t) = x \right]  }\\
	& = & \lim_{\epsilon \to 0, n \to \infty} \epsilon^k  (p_{n, n+1} \times 1 + p_{n, n} \times 0 + p_{n, n - 1} \times (-1)^k) \\
	& = & \lim_{\epsilon \to 0, n \to \infty} 
	\epsilon^k \{n \mathcal{O}(\epsilon^{-1} \delta t + \mathcal{O}(\delta t^2)\} \\ 
	& = & o(\delta t), 
\end{eqnarray*}
as $\delta t \to 0$, as required.
\end{proof}

%
%%%%%%%%%%%% Subection: Feller Diffusion as the limit of a BGW process %%%%%%%%%%%%
%
\subsection{Feller Diffusion as the limit of a BGW process}
\label{sec:BGWlimit}

Consider a BGW branching process with discrete generations $i = 0, 1, 2, \ldots$ and population size $Y(i)$ at generation $i$.  Assume the numbers 
of offspring per individual per generation are identically and independently distributed (i.i.d.) random variables $S_{ik}$, $k = 1, \ldots, Y(i)$, 
represented here by a generic random variable $S$ with $\Pr(S = 0) > 0$, $\mathbb{E}[S] =  \lambda, \Var(S) = \sigma^2$, and finite moments 
to all orders\footnote{
For certain applications, such as studying the coalescence of finite samples, the weaker assumption that $S$ has finite first and second moments 
will suffice \citep{Harris20}.
}.  
Then $Y(i + 1) = \sum_{k = 1}^{Y(i)} S_{ik}$.  We also assume that the process begins with a specified population size $Y(0) = y_0$.  

More specifically, consider the evolution of the subpopulation 
$\{M(i)\}_{i \ge 0},$
descended from a subset of of size 
$m_0 \le y_0$ of the initial population.  A diffusion limit is obtained for a process beginning at $M(0) = m_0$, $Y(0) = y_0$ by defining a time 
$t \in \mathbb{R}_{\ge 0}$ and scaled population $X(t)$ such that   
\begin{equation}	\label{BGWdiffLimit}
t = \frac{\sigma^2 i}{y_0}
\mbox{ for } i\in \mathbb{Z}_{\ge 0}, \quad X(t) = \frac{M(i)}{y_0}, 
\end{equation}
and by taking the limit 
$m_0, y_0 \rightarrow \infty$, $\lambda \rightarrow 1$, and $\sigma^2$ fixed, in such a way that 
\begin{equation}	\label{alphaDef}
\alpha := \frac{y_0 \log \lambda}{\sigma^2} \quad\text{ and }\quad x_0 := \frac{m_0}{y_0}
\end{equation}
remain fixed.  
Equivalently, we require that the mean number of offspring per parent is $\lambda = 1 + \alpha\sigma^2/y_0 + o(1/y_0)$ as $y_0 \to \infty$.   
%
%%	Proposition 2
%
\begin{proposition}\label{proposition2}
The limiting generator of the BGW process $X(t)$ with initial condition $X(0) = x_0$, as described above, is the Feller diffusion generator Eq.~(\ref{fellerGenerator}).  
\end{proposition}

That the Feller diffusion can be obtained as the limit of a BGW process is a classical result \citep[Example~5.9, p235]{Cox78}.
A proof consists of confirming that the functions $a$ and $b$ in Eq.~(\ref{aAndbLimits}),
with $\delta t = \sigma^2/y_0$ corresponding to an increment of one generation,
have the required limits.  
For a details in the context of the forward Kolmogorov equation with slightly different notational conventions see \citet[Section~2]{burden2016genetic}. 

%
%%%%%%%%%%%% Subsection: Feller Diffusion as the limit of a BD process %%%%%%%%%%%%
%
\subsection{Solution to the Feller Diffusion with initial condition $X(0) = x_0$}
\label{sec:FellerSolution}

\citet[][Lemma~9]{feller1951two} gives the solution for the density of $X(t)$ for a given initial condition.  The method of solution involves 
solving for the Laplace transform by integrating along characteristics, and is set out in detail in \citet[][Section~5.11]{Cox78}.  The density 
of $X(t)$ corresponding to the initial condition $X(0) = x_0$ is
\begin{eqnarray}	\label{FellersPoissonGammaMixture}
f_{X(t)}(x; \alpha, x_0) & = &  \mathcal{P}(0; x_0 \mu(t; \alpha))\delta(x) \\
	& & +\, \sum_{l = 1}^\infty \mathcal{P}(l; x_0 \mu(t; \alpha)) \frac{x^{l - 1}}{\beta(t; \alpha)^l (l - 1)!} e^{-x/\beta(t; \alpha)}, \quad x \ge 0,  \nonumber
\end{eqnarray}
where $\delta(x)$ is the Dirac delta function, 
$$
\mathcal{P}(l; \mu) =e^{-\mu} \frac{\mu^l}{l!} , \qquad l = 0, 1, 2, \ldots,
$$
and 
\begin{equation}	\label{muBetaDef}
\mu(t; \alpha) = \frac{2\alpha e^{\alpha t}}{e^{\alpha t} - 1}, \qquad \beta(t; \alpha) = \frac{e^{\alpha t} - 1}{2\alpha}.  
\end{equation}
We set $\mu(t; 0) = 2/t$ and $\beta(t; 0) =  t/2$.  
The expansion represents a Poisson-Gamma mixture, with an atom at zero of mass $\exp(-x_0 \mu(t; \alpha))$, the probability the population of 
extinction by time $t$, plus a continuous part for $x > 0$.  

The asymptotic limit as $t \to \infty$ of the atom at zero can be confirmed directly from a classical result stated in 
\citet[Theorem~1]{Athreya_1972} for the eventual extinction probability of a BGW process with a single founder.  In the notation of Section~\ref{sec:BGWlimit},  
for the process $\{M(i)\}_{i \ge 0}$ with $M(0) = m_0$, the eventual extinction probability is 
\begin{equation}	\label{BGWExtinction}
\mathbb{P}(M(\infty) = 0 \mid M(0) = m_0) = \zeta^{m_0}, 
\end{equation}
where $\zeta$ is the smallest non-negative root of the iterative equation $\zeta = G_S(\zeta) = \mathbb{E}[e^{\zeta S}]$.  
Moreover, this extinction probability is $< 1$ or $=1$ according as $\lambda = \mathbb{E}[S]$ is $> 1$ (supercritical case) or $\le 1$ 
(critical or sub-critical cases).   To obtain the near-critical limit of Proposition~\ref{proposition2}, note that the moment generating function satisfies 
\[
G_S(1) =1, \quad G'_S(1) = \lambda, \quad \text{and } G''_S(1) = \sigma^2 + \lambda^2 - \lambda > 0, 
\]
and hence 
\[
\zeta = 1 + \lambda(\zeta - 1) + \tfrac{1}{2} (\sigma^2 + \lambda^2 - \lambda)(\zeta - 1)^2 + O\left((\zeta - 1)^3\right), 
\]
as $\zeta \to 1$.
Solving for the smallest non-negative root, and noting that $\zeta$ is close to 1 for $\lambda$ close to 1, we have either $\zeta = 1$ 
if $\lambda \le 1$, or
\begin{eqnarray*}
\zeta & = & 1 + \frac{2(1 - \lambda)}{\sigma^2} + O\left( (\lambda - 1)^2 \right) \nonumber \\
	& = & 1 - \frac{2\log\lambda}{\sigma^2} + O\left( (\log\lambda)^2 \right) \nonumber \\
	& = & 1 - \frac{2\alpha}{y_0} + O\left( \frac{1}{y_0^2} \right), 
\end{eqnarray*}
if $\lambda > 1$, using Eq.~(\ref{alphaDef}) in the last line.  From Proposition~\ref{proposition2} and Eq.~(\ref{BGWExtinction}), the probability of 
eventual extinction for a supercritical Feller diffusion is  
\[
\lim_{y_0 \rightarrow \infty} \left(1 - \frac{2\alpha}{y_0} \right)^{x_0 y_0} = e^{-2\alpha x_0}, 
\]
agreeing with asymptotic limit of the point mass in Eq.~(\ref{FellersPoissonGammaMixture}) for the $\alpha > 0$ case.  The $\alpha \le 0$ case 
corresponds to the $\lambda \le 1$ stable solution.  

The interpretation of the continuous part of Eq.~(\ref{FellersPoissonGammaMixture}) is that the population at time $t$ is composed of independent 
families descended from $L = l$ founders at time $0$, where $L$ is Poisson with mean $x_0\mu(t; \alpha)$.
The size of each family is an independent exponential random variable with means $\beta(t; \alpha)$, 
as observed by \citet[Theorem~2.1(ii)]{o1995genealogy}.  
See also \citet{Harris20}.  
Fig.~\ref{fig:FellerSolution} represents a typical ancestry for a non-extinct population $X(t)$ with $l = 6$ ancestors. 

To understand the origin of this interpretation, consider the limiting process of a BGW process described in Subsection~\ref{sec:BGWlimit}.  The density of the 
distribution of the descendants of a single individual in the original population is approximated to $\mathcal{O}(1/y_0)$ by setting $m_0$ to 1 in Eq.~(\ref{alphaDef}) and hence 
$x_0 = 1/y_0$ in Eq.~(\ref{FellersPoissonGammaMixture}):
\begin{equation} \label{singleFamilyDensity}
f_\text{1-family}(x, t) = 
	\left(1 - \frac{\mu(t; \alpha)}{y_0} \right) \delta(x) + \frac{\mu(t; \alpha)}{y_0} \frac{1}{\beta(t; \alpha)} e^{-x/\beta(t; \alpha)} + {\cal O} \left(\frac{1}{y_0^2} \right) ,
\end{equation}
as $y_0 \rightarrow \infty$.  The probability of extinction of a given family is $1 - \mu(t; \alpha)/y_0 + {\cal O}(1/{y_0^2} )$, and 
the number of surviving families descended from an initial subpopulation of size $m_0$ at time $t$ as $m_0, y_0 \to \infty$ is therefore distributed as 
\begin{equation}	\label{NoOfFamiliesPois}
\text{Binom}\left(m_0, \frac{\mu(t; \alpha)}{y_0} + {\cal O} \left(\frac{1}{y_0^2} \right) \right) \rightarrow \text{Pois}(x_0 \mu(t; \alpha)) 
							\quad \text{as } y_0 \to \infty.  
\end{equation}
Eq.~(\ref{singleFamilyDensity}) also implies that, conditional on its survival, the limit in distribution of any family's size is exponential 
with rate $1/\beta(t; \alpha)$ as $y_0 \rightarrow \infty$.  

Since families evolve independently, conditional on precisely $l \ge 1$ families 
surviving to time $t$, the total population size is thus Gamma-distributed with rate parameter $1/\beta(t)$ and 
shape parameter $l$.  The interpretation of Eq.~(\ref{FellersPoissonGammaMixture}) for general $x_0$ as a sum over the number 
$l \ge 0$ of founders of the final population follows.  

As an interesting aside, we note that for the critical case, $\alpha = 0$, the classical \citet{yaglom1947certain} theorem for the quasi-stationary limit of 
a critical BGW process is implicit in Eq.~(\ref{FellersPoissonGammaMixture}).  In the notation of Section~\ref{sec:BGWlimit}, Yaglom's theorem states 
that \citep[see][Theorem~2]{Athreya_1972} in the critical case $\lambda = 0$, for which eventual extinction of the population is almost certain, 
$\lim_{i \to \infty} \mathbb{P}(Y(i)/i > z \mid Y(i) > 0) = e^{-2z/\sigma^2}$.  
In \citet[p191]{BurdenGriffiths23} it is shown that for a critical Feller diffusion, only the $l = 1$-family contribution survives the quasi-stationary limit, 
and that $\lim_{t \to \infty} \mathbb{P}(X(t)/t > w \mid X(t) > 0) = e^{-2w}$.  

For a general $\alpha \in \mathbb{R}$, the Laplace transform of $X(t)$ is 
\begin{eqnarray}	\label{expansion:0a}
\psi (\phi;t;\alpha,x_0) & = & \mathbb{E} \left[e^{-\phi X(t)} \right] \nonumber \\
	& = & \sum_{l=0}^\infty \mathcal{P}(l; x_0 \mu(t; \alpha))
				\left( 1 + \phi \beta(t;\alpha)\right)^{-l} \nonumber \\
	& = & \exp \left\{-x_0\mu(t;\alpha) \left(1 - \left( 1 + \phi \beta(t;\alpha)\right)^{-1}\right) \right\} \nonumber \\
	& = & \exp \left\{-\frac{\phi x_0 \mu(t;\alpha)\beta(t;\alpha)}{1+\phi \beta(t;\alpha)} \right\}. 
\end{eqnarray}

%
%%%%%%%%%%%% Section: Coalescence %%%%%%%%%%%%
%
\section{Coalescence}
\label{sec:Coalescence}

Let $A_n(s;t)$ be the number of ancestors of a sample of $n$ taken at the ``current" time $t$, at time $s$ back from the present.
$A_\infty(s;t)$ refers to the ancestors of the population (Seef Fig.~\ref{fig:AnDefinition}).  
$A_\infty(t;t)$ is the number of ancestors at time zero of the population at time $t$ and has a Poisson$(x_0 \mu(t;\alpha))$ distribution.

%%%%%%%%%%%% Figure: definitions of A_\infty and A_n  %%%%%%%%%%%%
\begin{figure}[t]
\begin{center}
\centerline{\includegraphics[width=1\textwidth]{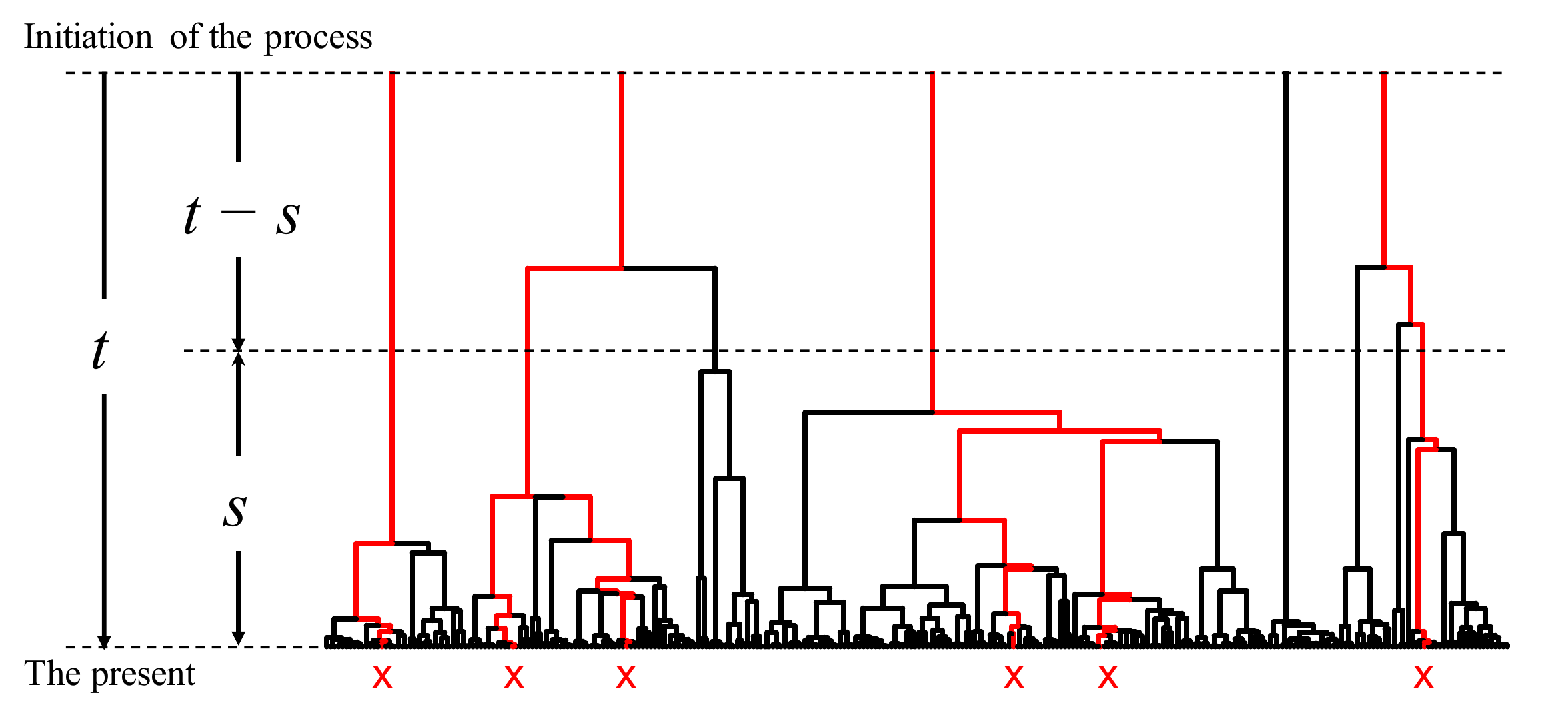}} % Plot constructed using <RRP_tree v4.R> 
\caption{Number of ancestors of a sample of size $n$, $A_n(s; t)$, at time $t - s$.  In this example, the number of ancestors of a sample 
of size $n = 6$, shown in red, is $A_6(s, t) = 4$.  The number of ancestors of the entire population is $A_\infty(s; t) = 8$. } 
\label{fig:AnDefinition}
\end{center}
\end{figure}

%
%%%%%%%%%%%% SubSection: Population Coalescence %%%%%%%%%%%%
%
\subsection{Population Coalescence}
\label{sec:PopCoalescence}

Our initial aim is to calculate the distribution of  $A_\infty(s;t)$.  We begin with the following definition.  
% Definition 1
\begin{definition} \citep[p378]{Johnson93} If $N \sim \text{Pois}(\nu)$ is a Poisson random variable with mean $\nu$ and probability generating function (pgf) 
$G_N(z)= e^{-\nu(1 - z)}$ and $M_i \sim \text{ShiftedGeom}(p)$, $0 \le p \le 1$, $i = 1, 2, \ldots$, are i.i.d.\ shifted geometric random variables, independent of $N$, 
with common probability function $\mathbb{P}(M = m) = (1 - p)p^{m - 1}$, $m = 1, 2, \ldots$ and pgf $G_M(z) = (1 - p)z(1 - pz)^{-1}$, then the random variable 
\[
Q = \sum_{i = 0}^N M_i, 
\]
where $M_0 := 0$, is said to be a \emph{geometric compound Poisson}  or \emph{P\'olya-Aeppli} random variable with parameters $(\nu, p)$.  
\end{definition}
The range of $Q$ is $\{0, 1, 2, \ldots\}$, with $\mathbb{P}(Q = 0) = \mathbb{P}(N = 0) = e^{-\nu}$.  
The corresponding pgf is 
\begin{equation}	\label{PApgf}
G_Q(z) = G_N(G_M(z)) = \exp \left\{ -\nu \frac{1 - z}{1 - pz} \right\}.
\end{equation}
%
%%	Theorem: distribution of A_\infty(s; t) = Polya-Aeppli
%
\begin{theorem}	\label{thm:PolyaAeppli}
$A_\infty(s;t)$ is a P\'olya-Aeppli random variable with parameters 
\begin{equation}	\label{nuandpDef}
\nu = x_0 \mu(t; \alpha) = \frac{2\alpha x_0 e^{\alpha t}}{e^{\alpha t} - 1}, \qquad p = p(s, t; \alpha) = \frac{e^{\alpha s} - e^{\alpha t}}{1 - e^{\alpha t}}.  
\end{equation}
\end{theorem}
\begin{proof}
Noting 
from Eq.~(\ref{NoOfFamiliesPois})
that, conditional on a scaled population $x$ at a time $s$ in the past, the number of  families descendant 
from founder lineages at $t - s$ is Poisson with mean $x\mu(s;\alpha)$, 
\begin{eqnarray*}	%\label{popanc:00}
\mathbb{P}\left(A_\infty(s;t)=k \right)
	& = & \int_0^\infty\mathbb{P}\left(A_\infty(s;t)=k, X(t-s) \in (x,x+dx)\right) \nonumber \\
	& = & \int_0^\infty\mathcal{P}\left(k; x\mu(s;\alpha)\right) f_{X(t-s)}(x; \alpha, x_0)dx.
\end{eqnarray*}
The pgf of $A_\infty(s; t)$ is, from Eq.~(\ref{expansion:0a}),  
\begin{eqnarray*}	%\label{pgf:00}
G_{A_\infty(s;t)}(z)
	& = & \int_0^\infty e^{-\mu(s;\alpha)x(1 - z)} f_{X(t - s)}(x; \alpha, x_0)dx \nonumber \\
	& = & \psi(\mu(s;\alpha)(1 - z), t - s; \alpha, x_0) \nonumber \\
	& = & \exp \left\{-\frac{\mu(s;\alpha)(1 - z) x_0 \mu(t - s;\alpha)\beta(t - s; \alpha)}{1 + \mu(s; \alpha)(1 - z) \beta(t - s;\alpha)}\right\} \nonumber \\
	& = & \exp \left\{ \frac{2\alpha x_0 e^{\alpha t}(1 - z)}{1 - e^{\alpha s} + (e^{\alpha s}-e^{\alpha t})(1 - z) } \right\} \nonumber \\
	& = & \exp\left\{ -x_0\mu(t; \alpha) \frac{1 - z}{1 - p(s, t; \alpha)z}  \right\}, 
\end{eqnarray*}
which is of the form of Eq.~(\ref{PApgf}).  
\end{proof}

It will be convenient below to consider the pgf of $A_\infty(s;t)$ conditional on $X(t) > 0$, i.e.\ on survival of the population at time $t$.  An equivalent condition is 
$A_\infty(s;t) > 0$ since the population at time $t$ is not extinct if and only if it has ancestors at time $t - s$.   
Since the {\em pgf} of any random variable $Q$ with 
range $\{0, 1, 2, \ldots\}$ is $G_Q(z) = \sum_{n = 0}^\infty  \mathbb{P}(Q = n) z^n$, the {\em pgf} of $Q$ conditioned on $Q > 0$ is 
\[
G_{Q\mid Q>0} (z) = \sum_{n = 1}^\infty  \mathbb{P}(Q = n \mid Q > 0) z^n = \frac{G_Q(z) - G_Q(0)}{1 - G_Q(0)}. 
\]
Hence
\begin{eqnarray}	\label{pgfAinfSurvival}
G_{A_\infty(s;t) \mid X(t) > 0}(z) & = & G_{A_\infty(s;t) \mid A_\infty(s;t) > 0}(z) \nonumber \\
& = & \frac{G_{A_\infty(s;t)}(z) - G_{A_\infty(s;t)}(0)}{1 - G_{A_\infty(s;t)}(0)} \nonumber \\
& = & \frac{\exp\left\{ -x_0\mu(t; \alpha) \frac{1 - z}{1 - p(s, t; \alpha)z}  \right\} - \exp\left\{ -x_0\mu(t; \alpha)\right\}}
				{1 - \exp\left\{ -x_0\mu(t; \alpha)\right\}}. \nonumber\\
\end{eqnarray}

%
%%%%%%%%%%%% Subsection: Sample Coalescence %%%%%%%%%%%%
%
\subsection{Sample Coalescence}
\label{sec:SampleCoalescence}

To calculate the coalescent distribution for a sample of size $n$ in terms of the coalescent distribution of the population we begin with the following 
lemma.  
%
%	lemma 1
%
\begin{lemma} \label{lemma:probAnGivenAinf}
Conditional on $A_\infty(s; t) = k > 0$, the probability that a sample of size $n > 0$ taken at time $t$ has $j \le k$ ancestors at time $t - s$ is 
\[
\mathbb{P}(A_n(s; t) = j \mid A_\infty(s; t) = k) = {k \choose j} \frac{n!}{k_{(n)}} {n - 1 \choose j - 1}, \qquad 1 \le j \le k, 
\]
where $k_{(n)} = k(k + 1)\cdots(k + n - 1)$ is the rising factorial.  
The distribution of $A_n(s;t)$ conditional on $A_\infty(s;t)$ is independent of the total population size at $t$.
\end{lemma}
\begin{proof}
Since at a given time each family size within the population descended from an individual is i.i.d.\ exponentially distributed, 
the distribution of relative family sizes of the population at time $t$ descended from the $k$ ancestors at time $t - s$ 
is Dirichlet$(\mathbf{1}_k)$, where the vector $\mathbf{1}_k = (1, \ldots, 1)$ has $k$ entries, independent of the total population size.  Hence the distribution of family 
sizes $(N_1, \ldots, N_k)$ in a sample of size $n$ is Dirichlet-multinomial$(n, \mathbf{1}_k)$ \citep[p80]{Johnson97}, 
\[
\mathbb{P}(\mathbf{N} = \mathbf{n} \mid A_\infty(s; t) = k) = \frac{n!}{k_{(n)}}. 
\]
This is a uniform distribution over the simplicial lattice $\{n_1, \ldots, n_k \ge 0: \sum_{i = 1}^k n_i = n\}$.  
Any point $(n_1, \ldots, n_k)$ in this lattice corresponding to exactly $j$ ancestors has $j$ non-zero components and $n - j$ zero components, 
with the non-zero components summing to $n$.  Thus the number of points in this lattice corresponding to 
$j$ ancestors is the number of ways of choosing the $j$ ancestors from $k$, times the number of points in the simplicial lattice 
$\{n_1, \ldots, n_j > 0: \sum_{i = 1}^j n_i = n \}$, that is, 
\[
{k \choose j} \times {n - 1 \choose j - 1},   
\]
and the result follows.  
\end{proof}

%
%%	Theorem: distribution of A_n(s; t)
%
\begin{theorem}	\label{thm:distribAn}
In a sample of size $n > 0$ taken at time $t$, the distribution of the number of ancestors at time $t - s$ is 
\begin{eqnarray} \label{samplelod:20}
\lefteqn{\mathbb{P}(A_n(s; t) = j)} \nonumber \\ 
	& = & n!{n-1\choose j-1} \sum_{k=j}^\infty  {k\choose j}\frac{1}{k_{(n)}}
\mathbb{P}\left(A_\infty(s; t) = k \mid A_\infty(s; t) > 0 \right), \quad j = 1, 2, \ldots, n. \nonumber \\
\end{eqnarray}
\end{theorem}
\begin{proof}
Because the sample size $n$ is assumed to be strictly positive, so the number of ancestors of 
the entire population at time $t - s$ must also be strictly positive.  Equivalently, conditioning on $A_\infty(s; t) >0$ 
has no effect on the probability of the event $A_n(s; t) = j$.  Then 
from Lemma~\ref{lemma:probAnGivenAinf}, for $1 \le j \le k$ and $0 < n < \infty$ we have the joint distribution
\begin{eqnarray*}
\lefteqn{\mathbb{P}(A_n(s; t) = j, A_\infty(s; t) = k)} \\ 
	& = & \mathbb{P}(A_n(s; t) = j, A_\infty(s; t) = k \mid  A_\infty(s; t) >0) \\
	& = & \mathbb{P}(A_n(s; t) = j \mid A_\infty(s; t) = k ) \mathbb{P}(A_\infty(s; t) = k\mid A_\infty(s; t) >0) \\
	& = & n!{n-1\choose j-1}{k\choose j}\frac{1}{k_{(n)}} \mathbb{P}(A_\infty(s; t) = k\mid A_\infty(s ;t) >0).   
\end{eqnarray*}
Note that the conditioning $A_\infty(s; t) > 0$ in this proof is essential to ensure that the joint distribution is correctly 
normalised: The range of $A_\infty(s; t)$ includes zero, corresponding to extinction of the population, whereas the range of $A_n(s; t)$ for $n > 0$, 
does not include zero.  Since $A_\infty(s; t) \ge A_n(s; t) > 0$, the event $A_\infty(s; t) = 0$ must be excluded.
Summing over $k \ge j$ then gives the required marginal distribution of $A_n(s; t)$.  
\end{proof}
%
%%%%%%%%%%%% Section: Coalescence as t -> infinity, subcritical %%%%%%%%%%%%
%
\section{Coalescence at fixed $s$ as $t \to \infty$: subcritical case}
\label{sec:CoalescenceTtoInfty}

The asymptotic limit as $t \to \infty$ of a subcritical branching process conditioned on survival of the population is known as the 
quasi-stationary limit \citep{lambert2007quasi}.  
Noting that for $\alpha < 0$, 
\[
\lim_{t \to \infty} \mu(t; \alpha) = 0, \qquad \lim_{t \to \infty} p(s, t; \alpha) = e^{-|\alpha|s}, 
\]
the pgf of the number of ancestors $A_\infty(s;t)$ of the population at time $s$ back in the quasi-stationary limit is, from Eq.~(\ref{pgfAinfSurvival}), 
\[
\lim_{t \to \infty}  G_{A_\infty(s;t) \mid A_\infty(s;t) > 0}(z) = \frac{(1 - e^{-|\alpha|s})z}{1 - z e^{-|\alpha|s}}.  
\]
This is the pgf of a shifted geometric distribution, 
\begin{equation}	\label{quasiStatAinftyDist}
\lim_{t \to \infty} \mathbb{P}(A_\infty(s;t) = k \mid A_\infty(s;t) > 0) = (1 - e^{-|\alpha|s})e^{-|\alpha|(k - 1)s}, \qquad k = 1, 2, \ldots.  
\end{equation}

Let $\{W_j\}_{j=2}^\infty$ be the quasi-stationary inter-coalescent waiting times when the population has $j$ ancestors, and let 
$T_k=\sum_{j=k}^\infty W_j$ be the time back to when there are first $k-1$ ancestors (see Fig.~\ref{fig:TnDefinition}). Then 
\begin{eqnarray*}	
\mathbb{P}(T_k > s) & = & \lim_{t \to \infty} \mathbb{P}(A_\infty(s;t) \geq k \mid A_\infty(s;t) > 0) \\
	& = & \sum_{j=k}^\infty (1-e^{-|\alpha| s})e^{-|\alpha|(j-1) s} \\
	& = & e^{-|\alpha|(k-1) s}.  
\end{eqnarray*}
Thus $T_k$ has an exponential distribution with rate $|\alpha|(k-1)$ and mean $1/(|\alpha|(k-1))$.
The mean coalescence times
\[
\mathbb{E}\big [W_k\big ] = \mathbb{E}\big [T_k - T_{k+1}] = \frac{1}{|\alpha|k(k-1)},
\]
are the same as in the Kingman coalescent apart from scale.  If a mutation occurs while $k$ ancestors then the mutant relative frequency 
$U$ in the population is Beta $(1,k-1)$ with density $(k-1)(1-u)^{k-2},\> u \in (0,1)$. \\

%%%%%%%%%%%% Figure: definitions of T_n and W_n, sub-critical  %%%%%%%%%%%%
\begin{figure}[t]
\begin{center}
\centerline{\includegraphics[width=0.55\textwidth]{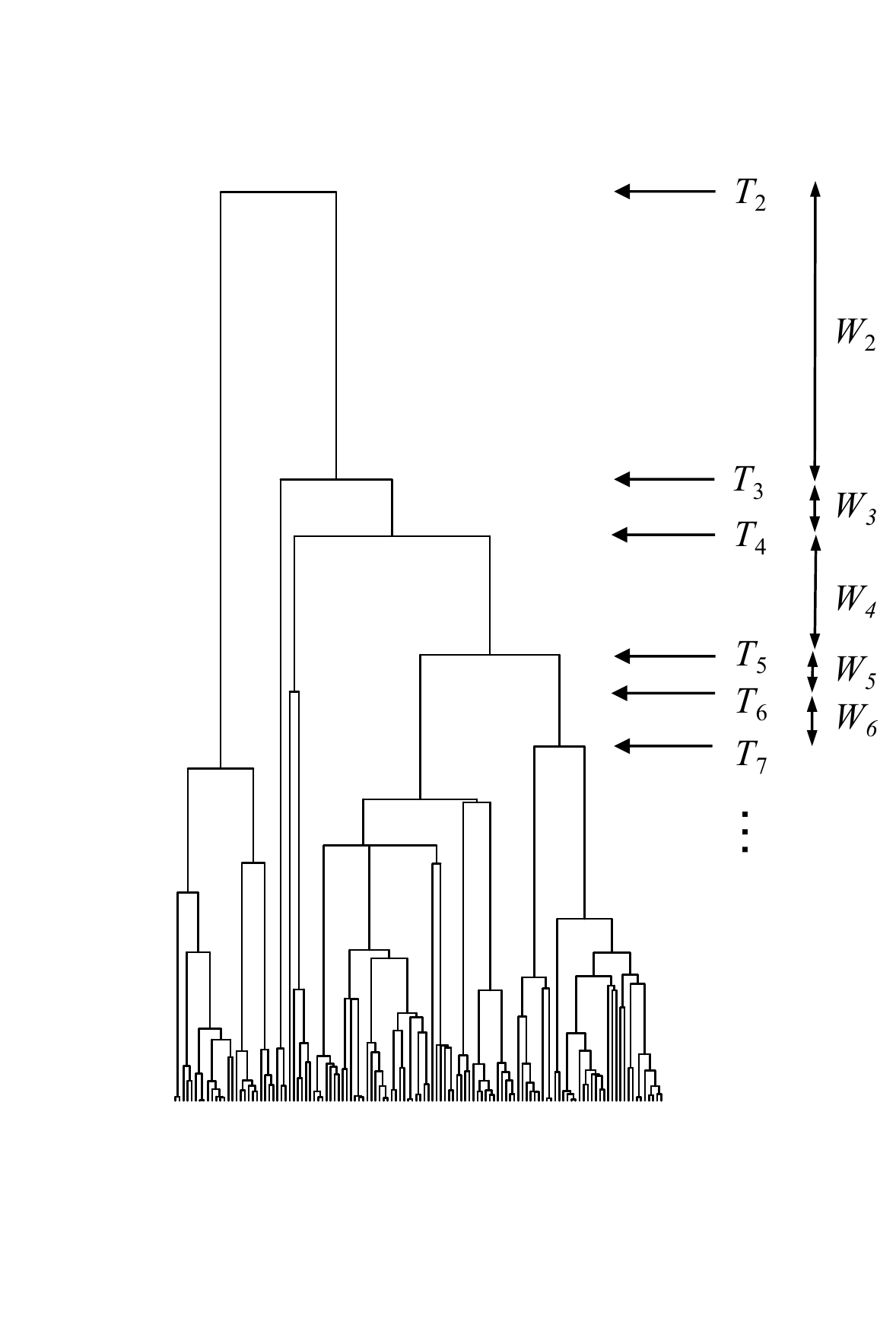}} % Plot constructed using <RRP_tree v4.R> 
\caption{Coalescent times $T_n$ measured back from the present, and inter-coalescent times $W_n = T_n - T_{n + 1}$.  } 
\label{fig:TnDefinition}
\end{center}
\end{figure}

%
%%%%%%%%%%%% Subsection: Sample coalescence  as t -> infinity, subcritical %%%%%%%%%%%%
%
\subsection{Quasi-stationary subcritical sample coalescent}
\label{sec:SubcriticalSampleCoalescence}

Substituting Eq.~(\ref{quasiStatAinftyDist}) into Eq.~(\ref{samplelod:20}) gives  
\[
\mathbb{P}(A_n(s; \infty) = j) = n!{n-1\choose j-1} \sum_{k=j}^\infty  {k\choose j}\frac{1}{k_{(n)}} (1 - e^{-|\alpha|s})e^{-|\alpha|(k - 1)s}.
\]
Writing 
\[
\frac{1}{k_{(n)}}=\frac{(k-1)!}{(k+n-1)!} = \frac{1}{(n-1)!}\cdot B(k,n)
=  \frac{1}{(n-1)!}\int_0^1u^{k-1}(1-u)^{n-1}du, 
\]
we have
\begin{eqnarray}	\label{dist:350}
\lefteqn{\mathbb{P}(A_n(s; \infty) = j)} \nonumber \\
	& = & n{n-1\choose j-1} \sum_{k=j}^\infty  {k\choose j}\int_0^1u^{k-1}(1-u)^{n-1}du\cdot
										(1-e^{-|\alpha| s})e^{-(k-1)|\alpha| s} \nonumber \\
	& = & n{n-1\choose j-1} (1-e^{-|\alpha| s})
				\int_0^1\frac{e^{-(j-1)|\alpha| s}}{(1 - ue^{-|\alpha| s})^{j+1}}u^{j-1}(1-u)^{n-1}du.  
\end{eqnarray}
We are interested in the expected sample inter-coalescent waiting times.  Begin with the integral 
\[
\mathcal{I}_j(u) = \int_0^\infty (1-e^{-|\alpha| s})\frac{e^{-(j-1)|\alpha| s}}{(1 - ue^{-|\alpha| s})^{j+1}}ds = \frac{1}{|\alpha| j(j-1)(1-u)^{j-1}}, 
\]
which follows from induction by checking that $\mathcal{I}_2(u) = \{2|\alpha|(1 - u)\}^{-1}$ and $\mathcal{I}_{j + 1}(u) = (j + 1)^{-1} \mathcal{I}_j'(u)$.  Then 
\begin{eqnarray*}
\mathbb{E}\left[W_j\right]
	& = & \mathbb{E}[T_j - T_{j + 1}] \\
	& = & \int_0^\infty \{ \mathbb{P}(T_j > s) - \mathbb{P}(T_{j + 1} > s) \} ds \\
	& = & \int_0^\infty \{ \mathbb{P}(A_n(s; \infty) \ge j) - \mathbb{P}(A_n(s; \infty) \ge j + 1) \} ds \\
	& = & \int_0^\infty \mathbb{P}\left(A_n(s; \infty) = j\right)ds \\
	& = & n{n-1\choose j-1}\frac{1}{|\alpha| j(j-1)}\int_0^1u^{j-1}(1-u)^{n-j}du \\
	& = & n{n-1\choose j-1}\frac{1}{|\alpha| j(j-1)}B(j,n-j+1) \\
	& = & \frac{1}{|\alpha| j(j-1)}.
\end{eqnarray*}

Once again this leads to the same expected waiting times as the Kingman coalescent, up to a scale.  
If neutral mutations between multiple alleles are included, in the limit of small mutation rates these waiting times lead to 
an identical distribution as that for a neutral stationary Wright-Fisher population \citep[Theorem~1]{BurdenGriffiths19a}.  
This equivalence has also been discovered by \citet[Corollary~1]{BurdenGriffiths23} using Laplace transform methods.  

In addition to its expectation, one can also deduce the distribution of $W_k = T_k - T_{k + 1}$.  Consider, with notation $w_k = t_k - t_{k + 1}$, 
\begin{eqnarray*}
\lefteqn{\mathbb{P}(T_k > t_k \mid T_{k + 1} = t_{k + 1}) = \mathbb{P}(A_k(W_k, \infty) = k)} \\
	& = & k(1 - e^{-|\alpha|w_k}) \int_0^1 \frac{e^{-|\alpha| (k - 1)w_k}}{(1 - u e^{-|\alpha|w_k})^{k + 1}} u^{k - 1}(1 - u)^{k - 1} du, 
\end{eqnarray*}
by substituting $s = w_k$, $n = j = l$ in Eq,~(\ref{dist:350}).  Since the conditional distribution does not depend on $t_{k + 1}$, 
\[
\mathbb{P}(W_k > w) = k(1 - e^{-|\alpha|w}) \int_0^1 \frac{e^{-|\alpha| (k - 1)w}}{(1 - u e^{-|\alpha|w})^{k + 1}} u^{k - 1}(1 - u)^{k - 1} du, 
\]
which differs from the exponential distribution waiting time for a Kingman coalescent.  The formula applies both to the coalescent tree of 
the entire population and to the coalescent tree of a finite sample.  
%
%%%%%%%%%%%% Section: Pop Coalescence as s -> 0 %%%%%%%%%%%%
%
\section{Population coalescence at fixed $t$ as $s \to 0$}
\label{sec:sToZero}

Recall from Theorem~\ref{thm:PolyaAeppli} that the number of ancestors at time $t - s$ of the population at time $t$ follows a P\'olya-Aeppli distribution, 
that is, the sum of a Poisson number of independent geometric random variables.  Write  
\[
A_\infty(s;t) = \sum_{j = 0}^{N(t)} M_j(s,t), 
\]
where $N(t) \sim \text{Pois}(x_0\mu(t; \alpha))$, $M_j(s,t) \sim \text{ShiftedGeom}(p(s,t))$, $j = 1, 2, \ldots$ with $p(s, t)$, defined by 
Eq.~(\ref{nuandpDef}), are independent, and $M_0(s, t) : = 0$.  The interpretation of Theorem~\ref{thm:PolyaAeppli} is that 
$N(t)$ is the number of founders at time $0$ of the entire population at time $t$, and $M_j(s, t)$ is the number of descendants 
of the $j$th founder alive at time $t - s$.  
For example, in the instance shown in Fig.~\ref{fig:AnDefinition}, $N(t) = 5$ and $(M_1(s, t), \ldots, M_5(s, t)) = (1, 2, 1, 1, 3)$.
Note that $M_j(s, t) > 0$ for $j > 0$ since the ancestral line of every member of the population at time $t$ 
passes through $t - s$ and that $A_\infty(s; t) = 0$ if and only if the population is extinct at time $t$.  

We next demonstrate that 
$M_j(s, t)$ can also be interpreted as a count of the descendants of a linear BD process with a single founder, conditioned on non-extinction.  
Consider a linear BD process $\{M(u)\}_{u \ge 0}$ with a single founder, $M(0) = 1$, birth rate $\hat{\lambda}$, and death rate $\hat{\mu}$,  
and transition probabilities as defined in Eq.~(\ref{linearBDtransitions}).     
Then \citet[p480]{Feller68} shows that  
\[
\mathbb{P}(M(u) = m) = \begin{cases}
	\hat{\mu}B(u) 											& m= 0 \\
	(1 - \hat{\lambda}B(u))(1 - \hat{\mu}B(u))(\hat{\lambda}B(u))^{m - 1}	& m = 1, 2, \ldots, 
	\end{cases}
\]
where 
\[
B(u) = \frac{1 - e^{(\hat{\lambda} - \hat{\mu})u}}{\hat{\mu} - \hat{\lambda}e^{(\hat{\lambda} - \hat{\mu})u}}.  
\]
Conditional on non-extinction of the process up to time $u$, 
\[
\mathbb{P}(M(u) = m \mid M(u) > 0) = (1 - \hat{\lambda}B(u))(\hat{\lambda}B(u))^{m - 1}, \qquad m = 1, 2, \ldots, 
\]
and so $(M(u) \mid M(u) > 0) \sim \text{ShiftedGeom}(\hat\lambda B(u))$.  Equating $\hat{\lambda}B(t - s)$ with $p(s, t; \alpha)$ defined by Eq.~(\ref{nuandpDef})
at fixed $t$ gives 
\[
\hat\lambda - \hat\mu = \alpha, \qquad \frac{\hat\mu}{\hat\lambda} = e^{-\alpha s}, 
\]
or 
\begin{equation} \label{lambdaHatMuHatSoln}
\hat\lambda(s) = \frac{\alpha}{1 - e^{-\alpha s}}, \qquad \hat\mu(s) = \frac{\alpha e^{-\alpha s}}{1 - e^{-\alpha s}}.  
\end{equation}

\citet{Nee94} define the {\em reconstructed process} (RP) of a linear BD process with birth rate $\hat\lambda$ and death rate $\hat\mu$ 
stopped at time $T$ after initiation to be the time-inhomogeneous pure-birth process constructed by pruning lineages which have not survived 
to time $T$.  Denote by $\mathcal{RP}(T, \hat\lambda, \hat\mu)$ the set of RPs for a given $T$, $\hat\lambda$ and $\hat\mu$.   
Elements of this set differ by their initial conditions $M(0)$.  
With this definition we have the following lemma, illustrated in 
Fig.~\ref{fig:RPLemma}.  
%
%%	Lemma: Partial coalescent tree of a FD is the RP of a linear BD process
%
\begin{lemma}	\label{lemma:RPlemma}
For any $0 < s < t$, that part of the coalescent tree of a parameter-$\alpha$ Feller diffusion $\{X(u)\}_{0 \le u \le t}$ 
lying within the interval $[0, t - s]$ is generated by a process in $\mathcal{RP}(t - s, \hat\lambda(s), \hat\mu(s))$, where $\hat\lambda(s)$ and $\hat\mu(s)$ 
are defined by Eq.~(\ref{lambdaHatMuHatSoln}).  
\end{lemma}

%%%%%%%%%%%% Figure: Illustration of Lemma 2, RP of BD process is Feller diffusion coalescent tree %%%%%%%%%%%%
\begin{figure}[t]
\begin{center}
\centerline{\includegraphics[width=0.8\textwidth]{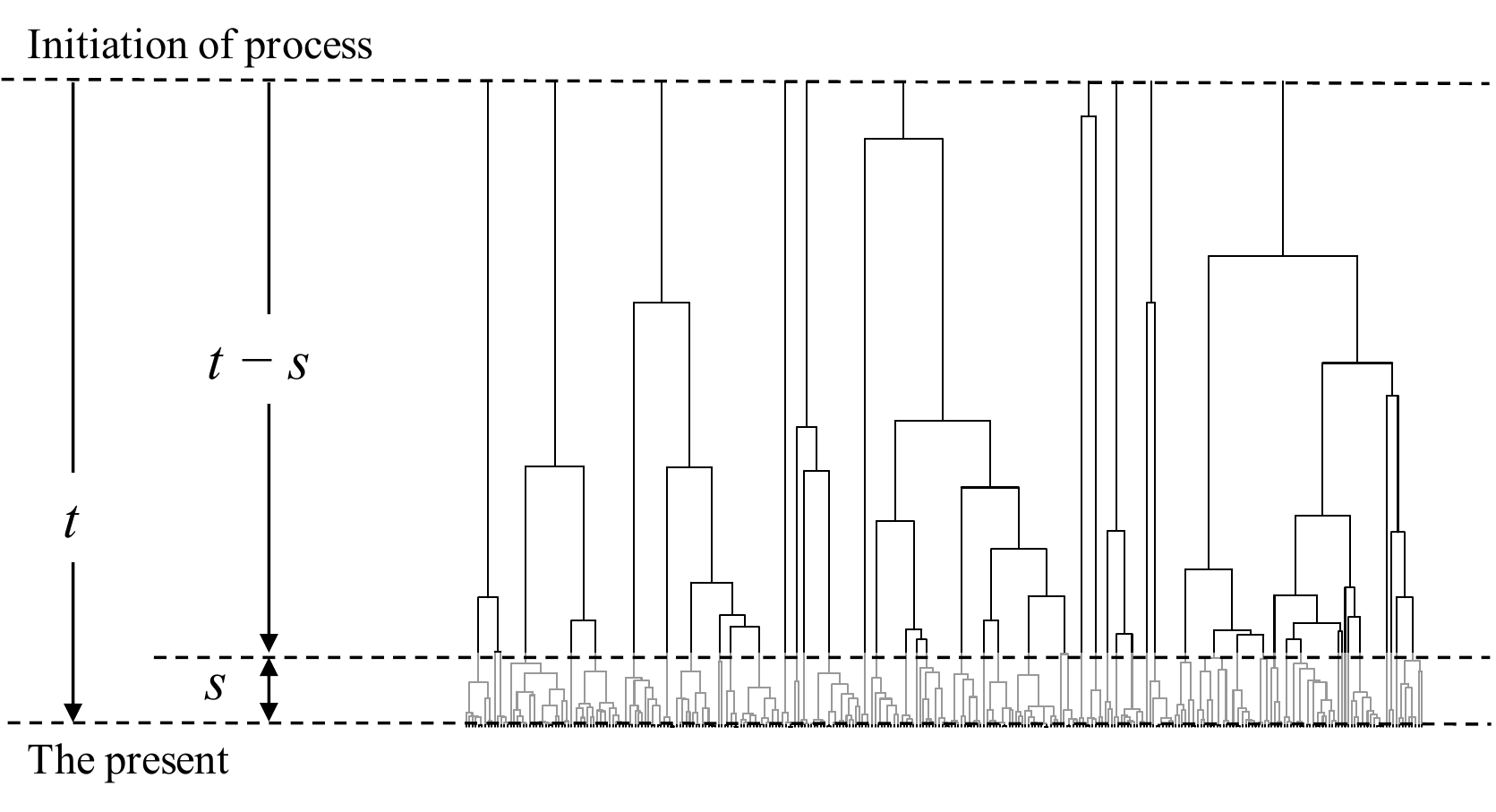}} % Plot constructed using <RRP_tree v4.R> 
\caption{Illustration of Lemma~\ref{lemma:RPlemma}: That part of the coalescent tree of the Feller diffusion $\{X(u)\}_{0 \le u \le t}$ lying in the interval $[0, t - s]$ 
 is generated by an element of $\mathcal{RP}(t - s, \hat\lambda(s), \hat\mu(s))$.} 
\label{fig:RPLemma}
\end{center}
\end{figure}
%%%

\begin{proof}
The inhomogeneous birth rate of any process in $\mathcal{RP}(T, \hat\lambda, \hat\mu)$ at time $u$ since initiation is equal to \citep[Eq.~(1.5)]{Ignatieva20}
\begin{eqnarray*}
\lefteqn{\lambda^{\rm eff}(u; T, \lambda, \mu)  } \\
	& = & \hat\lambda \times \mathbb{P}(\text{a single lineage born at time $u$ is not extinct by time $T$}) \\
	& = & \hat\lambda \times (1 - \hat\mu B(u)) \\
	& = & \hat\lambda \times \frac{\hat\lambda - \hat\mu}{\hat\lambda - \hat\mu e^{-(\hat\lambda - \hat\mu)(T - u)}}, \qquad 0 \le u \le T.  
\end{eqnarray*}
Substituting $t - s$ for $T$ and Eq.~(\ref{lambdaHatMuHatSoln}) for $\hat\lambda$ and $\hat\mu$, the inhomogeneous birth rate for any process 
in $\mathcal{RP}(t - s, \hat\lambda(s), \hat\mu(s))$ is 
\begin{eqnarray}	\label{lambdaEffDef}
\lambda^{\rm eff}(u; t - s, \hat\lambda(s), \hat\mu(s)) 
			& = & \frac{\hat\lambda(s) - \hat\mu(s)}{1 - (\hat\mu(s)/\hat\lambda(s))e^{-(\hat\lambda(s) - \hat\mu(s))(t - s - u)}} \nonumber\\
			& = & \frac{\alpha}{1 - e^{-\alpha(t - u)}},    \qquad 0 \le u \le t - s.   
\end{eqnarray}
The important point to note is that this inhomogeneous birth rate does not depend on $s$, and hence elements of the set of reconstructed processes 
$\{\mathcal{RP}(t - s, \hat\lambda(s), \hat\mu(s)) : 0 \le s < t\}$ differ only to the extent that we choose to stop the processes at different times $s$ short of the 
limiting time $t$. Since the running parameters $\hat\lambda(s)$ and $\hat\mu(s)$ are constructed from a 
process which generates as its RP the number of ancestors $A_\infty(t - u, t)$ of the entire population of a Feller diffusion with parameter $\alpha$, the result follows.
\end{proof}

We are now in a position to construct the coalescent tree of a Feller diffusion as the limiting case of the RP of a linear BD process.  
%
%%	Theorem: FD coalescent tree is a pure-birth process with given inhomogeneous birth rate
%
\begin{theorem}	\label{thm:coalescentTreeIsAnRP}
The coalescent tree of a Feller diffusion with parameter $\alpha$, observed at time $t$ since initiation, is a time-inhomogeneous pure-birth 
process with birth rate at time $u$ since initiation equal to 
\[
\lambda^{\rm coal}(u; t) := \frac{\alpha}{1 - e^{-\alpha(t - u)}}, \qquad 0 \le u < t.  
\]
\end{theorem}
\begin{proof}
Consider the limit $s \to 0$ in Fig.~(\ref{fig:RPLemma}).  From Eq.~(\ref{lambdaHatMuHatSoln}), 
\[
\hat\lambda(s) = s^{-1} + \tfrac{1}{2} \alpha + \mathcal{O}(s), \quad
\hat \mu(s)        =s^{-1} - \tfrac{1}{2} \alpha + \mathcal{O}(s), \qquad \text{as } s \to 0. 
\]
Identify $\tfrac{1}{2}s$ with $\epsilon$ in Proposition~\ref{proposition1} and define a random variable 
\begin{equation}	\label{MtoXlimit}
X_s(t) = \tfrac{1}{2}s M(t - s), 
\end{equation}
where $\{M(u)\}_{0 \le u \le t - s}$ is a birth-death process with rates $\hat\lambda(s)$, $\hat\mu(s)$.  Then as $s \to 0$ the limiting generator of the process $X_s(t)$ 
is that of  a Feller diffusion with parameter $\alpha$.  The required result then follows immediately from Lemma~\ref{lemma:RPlemma} and that 
$\lambda^{\rm coal}(u; t) = \lambda^{\rm eff}(u; t - s, \hat\lambda(s), \hat\mu(s))$ is independent of $s$.  
\end{proof}

In the context of population genetics, the Feller diffusion is usually approached as a limit of a BGW process as described in Section~\ref{sec:BGWlimit}.  
The above procedure provides an alternative view in terms of that part of the population's coalescent tree at times earlier than $t - s$.  
As $s \to 0$ the effective birth rate $\lambda^{\rm coal}(t - s; t) \to \infty$, consistent with the coalescent tree ``coming down from infinity''.  
The running parameters $\hat\lambda(s)$, $\hat\mu(s)$ also become infinite as $s \to 0$ while their difference $\alpha$ remains fixed.  
Thus a Feller diffusion can be seen as the limit of a linear BD process in which births and deaths occur at an arbitrary high rate.  

It remains to match the scale $s$ of the BD limit with the scale $y_0$ and parameters $\log\lambda$ and $\sigma^2$ of the BGW limit
in such a way that the limiting process is the same Feller diffusion for both limits.  
From Eqs.~(\ref{BGWdiffLimit}), (\ref{alphaDef}) and (\ref{MtoXlimit}) we have 
\begin{equation}	\label{sMeaning}
\tfrac{1}{2}s = \frac{1}{y_0} = \frac{\log\lambda}{\alpha\sigma^2}. 
\end{equation}
Implicit in this matching is the notion that the ancestry of a large-population, near-critical BGW process is well approximated by 
the coalescent tree of the corresponding linear BD process.  
%
%%%%%%%%%%%% Section: Coalescent tree as a reversed-reconstructed process %%%%%%%%%%%%
%
\section{The coalescent tree of a super-critical Feller diffusion as a reversed reconstructed process}
\label{sec:RRProcess}

\cite{Gernhard08} and \cite{Ignatieva20} generate the coalescent tree of a supercritical 
linear BD process as a RRP.  The RRP is an inhomogeneous pure-death process which runs backwards in time from the present tracing the 
ancestral coalescent tree.  Construction of the RRP from a BD process consists of first constructing 
the RP for a fixed time since initiation with a single ancestor, and then imposing an improper uniform prior on the time since initiation 
and conditioning on the present number of leaves~\citep{Aldous05,Wiuf18}.  By effecting a time rescaling of the 
RRP to a rate-1 pure-death Yule process, ~\cite{Ignatieva20} reproduce earlier results of \citet[Theorem~4.1]{Gernhard08} 
for the distribution of coalescent times for a supercritical BD process with constant birth and death rates.  

%%%%%%%%%%%% Figure: Illustration of Lemma 2, RP of BD process is Feller diffusion coalescent tree %%%%%%%%%%%%
\begin{figure}[t]
\begin{center}
\centerline{\includegraphics[width=\textwidth]{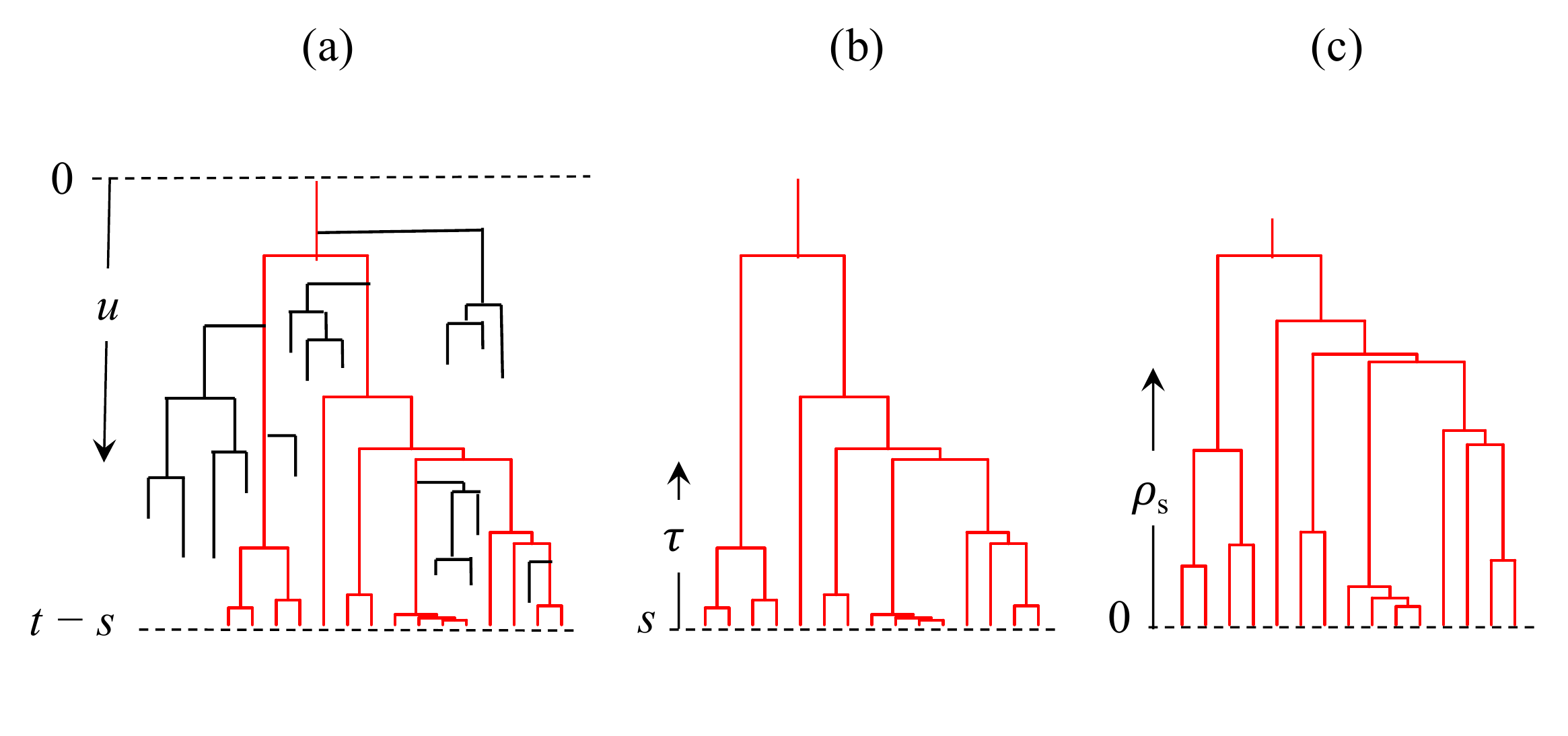}} % Plot constructed using <RRP_tree v4.R> 
\caption{Construction of the RRP from the RP of a linear BD process, and subsequent mapping to a rate-1 pure death process:  (a) a 
linear BD process process defined over the interval $0 \le u \le t - s$, with the corresponding RP shown in red;  (b) the constructed RRP over 
$s \le \tau < \infty$; and (c) the corresponding rate-1 pure death process over the interval $0 \le \rho_s < \infty$.  } 
\label{fig:RRP}
\end{center}
\end{figure}
%%%

Below we construct the coalescent tree for a super-critical Feller diffusion conditioned on $X(t) = x$ with an improper uniform prior
on initiation of a process with a single ancestor.  The construction is based on the $s \to 0$ limit of the RRP derived from the RP of a linear BD 
process on the interval $[0, t - s]$ with birth-rate $\hat\lambda(s)$ and death rate $\hat\mu(s)$.  The RP of a linear BD process is represented in 
Fig~\ref{fig:RRP}(a).  To construct the RRP represented in Fig.~\ref{fig:RRP}(b) we introduce a parameter $\tau = t - u$, $s \le \tau < \infty$, 
measured backwards from a time $s$ in the past before the present and assume an improper uniform prior on the time since initiation of the 
forward process.  From Eq.~(\ref{lambdaEffDef}), the effective death rate of this RRP is 
\[
\mu_{\rm eff}(\tau) = \frac{\alpha}{1 - e^{-\alpha\tau}}, \qquad s \le \tau < \infty.  
\]
The RRP is conditioned on a boundary at $\tau = s$ and is allowed to run backwards until the birth of the earliest ancestor.  

Following \citet[Section~2.1]{Ignatieva20}, the inhomogeneous pure-death process can be mapped onto a time-reversed rate-1 Yule process, 
represented in Fig~\ref{fig:RRP}(c).  To this end, define
a rescaled time $\rho_s(\tau)$ such that, if the random variable $T$ is the time of death of an individual reversed lineage starting at time $s$, 
then $\mathbb{P}(T > \tau) = e^{-\rho_s(\tau)}$ for $\tau \ge s$ and $e^{\rho_s(s)} = 1$.  Thus 
\[
\mu_{\rm eff}(\tau)d\tau = \frac{\mathbb{P}(\tau \le T < \tau + d\tau)}{\mathbb{P}(T \ge \tau)} = 
			\frac{e^{-\rho_s(\tau)} - e^{-\rho_s(\tau + d\tau)}}{e^{-\rho_s(\tau)}} = \rho_s'(\tau) d\tau, 
\]
which, together with the initial condition gives 
\[
\rho_s(\tau) = \int_s^\tau \mu_{\rm eff}(\xi) d\xi = \int_s^\tau \frac{\alpha e^{\alpha\xi}}{e^{\alpha\xi} - 1}d\xi= \log\frac{e^{\alpha\tau} - 1}{e^{\alpha s} - 1}.  
\]
The mapping to a rate-1 Yule process inverts to 
\begin{equation}	\label{tauOfRho}
\alpha\tau = \log\left(1 + (e^{\alpha s} - 1)e^{\rho_s} \right), \qquad 0 \le \rho_s < \infty.  
\end{equation}

As an aside we note that the transformation to $\rho_s(\tau)$ accords with a remark of \citet[p425]{o1995genealogy} regarding the limit of a near-critical continuous-time BGW process, 
once differences in notation and scaling, namely setting $\tau$ to unity and replacing $(\tau - s)/\tau$ with O'Connell's $r$, 
are accounted for.  
A similar time change also appears in \citet[below Eq.~(4)]{Harris20}.

Implicit in the above construction is an assumption that the RRP converges to a single ancestor.  
For the sub-critical case, $\alpha = \hat\lambda(s)  - \hat\mu(s) < 0$, one has that 
\[
\lim_{\tau \to \infty} \mathbb{P}(T > \tau) = \lim_{\tau \to \infty} e^{-\rho_s(\tau)} = 1 - e^{-|\alpha|s}, 
\]
and there is a finite probability that an ancestral lines 
will never converge.  For this reason the above construction only applies for the super-critical case $\alpha > 0$.  

It would be nice to dispense with the scale $s$ at this point by sending it to zero, but this entails beginning the time-reversed Yule 
process with an infinite number of particles at an infinite time in the past.  This is a manifestation of the effect that the coalescent tree of a 
diffusion process comes down from infinity.  However, recall from Eq.~(\ref{sMeaning}) that the scale $s$ also has meaning in terms of the 
BGW process underlying the Feller diffusion limit.  With this interpretation, the parameter $\alpha$ and 
currently observed scaled population $X = x$ are related to the current (assumed arbitrarily large) physical population $M = y$ 
of an underlying BGW process with mean and variance $\lambda$ and $\sigma^2$ of the number of offspring per generation via 
Eqs.~(\ref{MtoXlimit}) and (\ref{sMeaning}) by 
\begin{equation}	\label{xToYScale}
x = \tfrac{1}{2}s y = \frac{\log \lambda}{\alpha \sigma^2} y. 
\end{equation} 
Note that the initial conditions of the Feller diffusion process have been subsumed into the uniform prior on the time since initiation of the ancestral tree.  

Figure~\ref{fig:RRP_tree_n500_s0.001} shows (a) a plot of a rate-1 pure death process and (b) the same tree mapped from the timescale 
$\rho_s$ onto the timescale $\alpha\tau$.  This is a coalescent tree of the entire population with the branches 
over the recent short interval $0 \le \tau < s$ truncated.  

%%%%%%%%%%%% Figure: reverse reconstructed tree rho and tau  %%%%%%%%%%%%
\begin{figure}[t]
\begin{center}
\centerline{\includegraphics[width=0.7\textwidth]{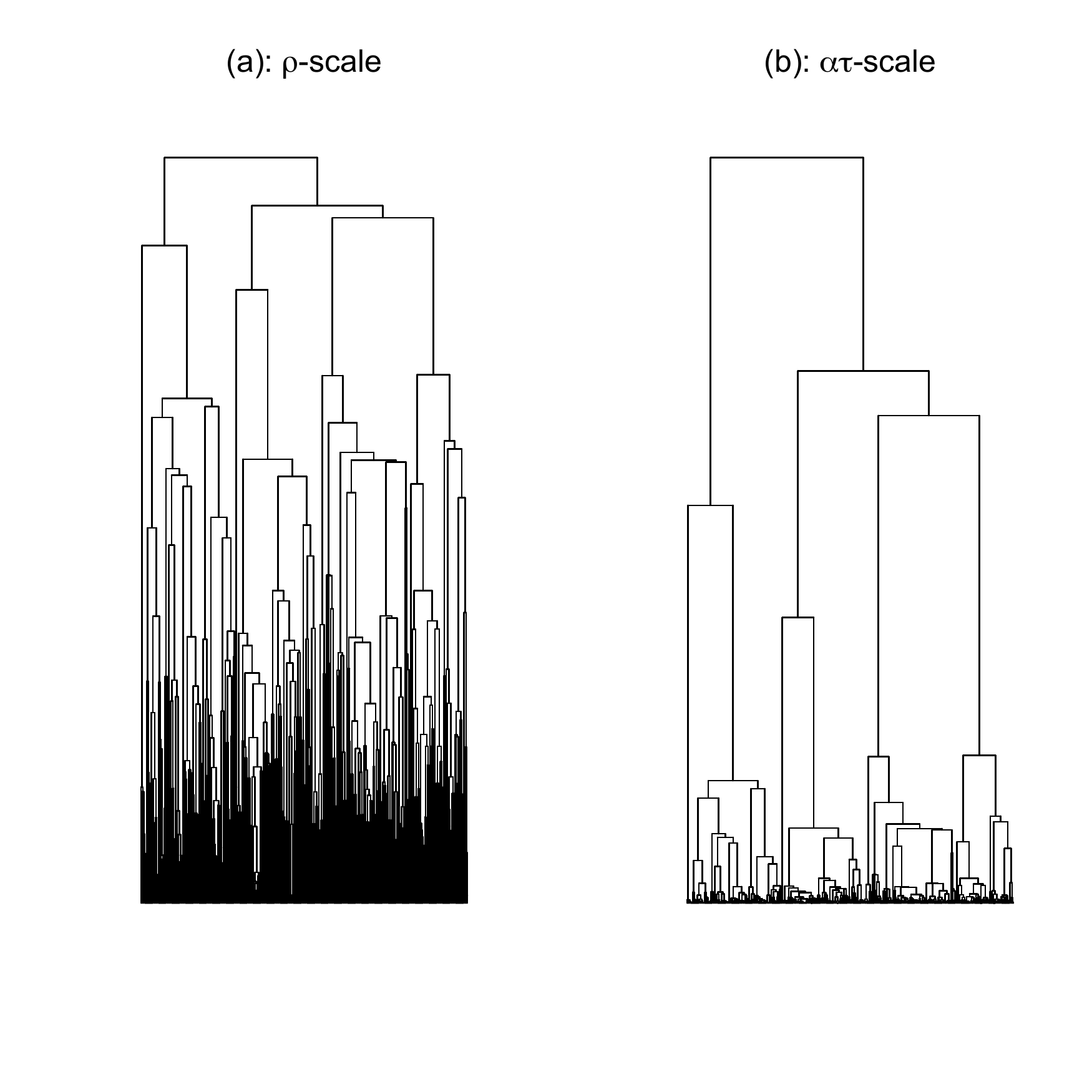}} % Plot constructed using <RRP_tree v4.R> 
\caption{(a) Tree constructed from a rate-1 pure-death process in time $\rho_s$ starting with $500$ leaves at time $\rho_s = 0$, 
and (b) the corresponding reverse-reconstructed tree mapped from the time scale $\rho_s$ onto a time scale $\alpha\tau$ for $\tau > s$ using 
Eq.~(\ref{tauOfRho}) with $\alpha s = 0.001$. } 
\label{fig:RRP_tree_n500_s0.001}
\end{center}
\end{figure}

%
%%%%%%%%%%%% Section: Coalescent times in the RRP of a super-critical Feller diffusion %%%%%%%%%%%%
%
\section{Coalescent times in the RRP of a super-critical Feller diffusion}
\label{sec:RRPCoalescentTimes}
The following theorem gives the distribution of coalescent times as defined in Fig.~\ref{fig:TnDefinition} and of the number of ancestors at time $\tau$ before present.  
%
%%	Theorem: number of ancestors in a reverse reconstructed supercritical Feller process is Poisson
%
\begin{theorem} \label{PoissonAncestors}
Let $\cdots < T_j <  \cdots < T_3 < T_2 < T_1$ be coalescence times in the coalescent tree of the reverse reconstructed process 
of a super-critical Feller diffusion beginning with population size $x$. The coalescent tree is assumed to be extended back in time so that there is a single 
ancestor. Then the coalescence times are distributed as a non-homogeneous Poisson process back in time with rate $\frac{d}{d\tau}\{-x/\beta(\tau)\}$, 
where $\beta(\tau)$ is defined by Eq.~(\ref{muBetaDef}), and
\begin{equation}
\mathbb{P}\left(N(\tau) = j \right) = \frac{1}{j!}
\left(\frac{x}{\beta(\tau)}\right)^j
\exp \left\{-\frac{x}{\beta(\tau)}\right\},\> \tau > 0,\ j=1,2,\ldots,
\label{Poisson:260}
\end{equation}
where $N(\tau)$ is be the number of ancestors at time $\tau$ back.  
\end{theorem}
\begin{proof}
Suppose that the current time is $t$ from the initiation of the population with one founder.
We argue conditional on $M(t-s) = m$ where $\{M(u)\}_{0 \le u \le t - s}$ is the RP in Fig.~\ref{fig:RRP}(a), 
then let $m \to \infty, s\to 0, \frac{1}{2}sm = x$, conditioning on the diffusion population size at $t$ being $x$ in the limit.
Coalescent times before taking the limit form a reverse Markov chain with
\begin{eqnarray}
\mathbb{P}\left(T_{j} > \tau_j\mid T_{j+1} = \tau_{j+1}\right) 
	& = & \exp \left\{-j\int_{\tau_{j+1}}^{\tau_j}\mu_{\text{eff}}(\xi)d\xi \right\} \nonumber\\
	& = & \exp \left\{-j\log \left(\frac{e^{\alpha \tau_j} - 1}{e^{\alpha \tau_{j+1}} - 1} \right) \right\} \nonumber\\
	& = &\left(\frac{e^{\alpha \tau_{j+1}} - 1}{e^{\alpha \tau_j} - 1} \right)^j,\  j = m - 1,\ldots, 1 \label{T:250} \\
\mathbb{P}\left(T_{m} > \tau_{m}\right) 
	& = & \left(\frac{e^{\alpha s} - 1}{e^{\alpha \tau_{m} }- 1} \right)^{m},\ \tau_{m} > s.\label{T:260}
\end{eqnarray}
The distribution of $T_j$ given $T_{j+1}=\tau_{j+1}$ is seen to be that of the minimum of $j$ independent random variables with distribution function $F$ such that
\[
1-F_{\tau_{j+1}}(\tau) = \frac{e^{\alpha \tau_{j+1}} - 1}{e^{\alpha \tau} - 1},
 \quad \tau > \tau_{j+1}.
\]
That is, $T_{m},T_{m-1}, \ldots T_1$ are distributed as order statistics from small to large in an independent sample of $m$ from a distribution with
\begin{equation*}
1-F(\tau) = \frac{e^{\alpha s} - 1}{e^{\alpha \tau} - 1}, \quad \tau > s,
\end{equation*} 
and corresponding density
\[
f(\tau) = \frac{\alpha e^{\alpha\tau}(e^{\alpha s} - 1)}{(e^{\alpha \tau} - 1)^2}, \quad \tau > s,
\]
because
\[
\mathbb{P}\left(T_{j} > \tau_j\mid T_{j+1} = \tau_{j+1}\right) 
= \left( \frac{1- F(\tau_j)}{1 - F(\tau_{j+1})}\right)^j,
\]
agreeing with (\ref{T:250}).
This is the probability that there are $j$ sample points larger than $\tau_j$ in the distribution truncated below at $\tau_{j+1}$. 
The dependence in $s$ appears in the distribution of $T_m$ in (\ref{T:260}).
Consider the pre-limit density of the $r$ maximal points $T_1 > T_2 > \cdots >T_r$ for fixed $r$,
\[
\frac{m!}{(m-r)!}f(\tau_1)\cdots f(\tau_t)F(\tau_r)^{m-r},\ \tau_1>\tau_2>\cdots > \tau_r.
\]
Now
\[
mf(\tau) \to 
 \frac{2x\alpha^2 e^{\alpha\tau}}{(e^{\alpha \tau} - 1)^2} = \frac{d}{d\tau} \left(\frac{-x}{\beta(\tau)}\right), 
 \]
 and
 \[
  F(\tau)^{m-r}
 \to  \exp \left\{-\frac{2\alpha x}{e^{\alpha \tau}-1}\right\} = \exp \left\{\frac{-x}{\beta(\tau)}\right\}. 
 \]
 The limit density of $T_1,\ldots ,T_r$ is therefore
 \[
 \left( \prod_{j=1}^r  \frac{d}{d\tau_j} \left\{\frac{-x}{\beta(\tau_j)}\right\} \right)
 \exp \left\{\frac{-x}{\beta(\tau_r)}\right\},\quad \tau_1 > \tau_2 > \cdots > \tau_r>0.
 \]
 These are finite-dimensional densities of points in the non-homogeneous Poisson process in the statement of the theorem, so the proof is completed. 
 The Poisson distribution (\ref{Poisson:260}) follows, however it can be argued directly. In the pre-limit the probability that in a sample of $m$ there are 
 $m-j$ points less than $\tau$ and $j$ points greater than $\tau$ is 
 \[
 {m\choose j}\left(1 - F(\tau)\right)^jF(\tau)^{m-j}
 \]
 whose limit is Eq.~(\ref{Poisson:260}).
 
 \end{proof}
%
%%	Corollary: number of ancestors at time tau in the past of a sample of size n
%
\begin{corollary}	\label{corollary:sampleCoalescence}
Let $N_n(\tau)$ be the number of ancestors of a sample of $n$ taken from a super-critical Feller
diffusion with current population size $x$. The coalescent tree of the sample is assumed
to be extended back in time so that there is a single ancestor. Then for $j = n, n-1, \ldots, 1$
\begin{eqnarray}	\label{sample:235}
\mathbb{P}\left(N_n(\tau) = j\right) 
	& = & \frac{n_{[j]}}{n_{(j)}j!} \left(\frac{x}{\beta(\tau)}\right)^j \exp \left\{-\frac{x}{\beta(\tau)}\right\}
														M\left( j, j+n, \frac{x}{\beta(\tau)}\right) \nonumber\\
	& = & \frac{1}{(j-1)!}{n\choose j} \left(\frac{x}{\beta(\tau)}\right)^j
						\int_0^1 \exp \left\{-(1 - v)\frac{x}{\beta(\tau)}\right\} v^{j - 1}(1 - v)^{n - 1} dv, \nonumber\\
\end{eqnarray}
where $M$ is Kummer's confluent hypergeometric function defined by~\citep[Eq.~(13.1.2)]{Abramowitz:1965sf} 
\[
M(a,b,z) = \sum_{k=0}^\infty\frac{a_{(k)}}{b_{(k)}}\frac{z^k}{k!}.
\]
Let $\left\{W^{(n)}_j\right\}_{j=2}^n$ be the inter-coalescent times in a sample of $n$.  
The mean inter-coalescent times are
\begin{equation}	\label{sample:236}
\mathbb{E}\left[W^{(n)}_j\right] = \frac{2(2\alpha)^{j-1}x^j}{(j-1)!}{n\choose j}
								\int_0^1v^{j - 1}(1 - v)^{n - 1} \int_0^\infty\frac{u^{j - 1}}{1 + u}e^{-2\alpha x(1 - v)u} dudv.
\end{equation}
\end{corollary}
\begin{proof}
From Lemma~\ref{lemma:probAnGivenAinf}
\begin{equation}	\label{NN:25}
\mathbb{P}(N_n(\tau) = j \mid N(\tau) = l) = {l \choose j} \frac{n!}{l_{(n)}} {n - 1 \choose j - 1}, \qquad 1 \le j \le l, 
\end{equation}
because $N_n(\tau)$ is independent of the total current population size $X$ conditional on $N_\infty(\tau)$.  The distribution (\ref{NN:25}) 
also only depends on the ancestry at time $\tau$ back from present time, not on the fact that the process is initiated from one founder.

Then from Theorem~\ref{PoissonAncestors} the unconditional probability
\begin{eqnarray*}
\mathbb{P}\left(N_n(\tau) = j\right) & =& \sum_{l=j}^\infty{l \choose j} \frac{n!}{l_{(n)}} {n - 1 \choose j - 1}
								\frac{1}{l!}\left(\frac{x}{\beta(\tau)}\right)^l  \exp \left\{-\frac{x}{\beta(\tau)}\right\} \\
				& = & \frac{1}{j!}\left(\frac{x}{\beta(\tau)}\right)^j \exp \left\{-\frac{x}{\beta(\tau)}\right\} \frac{(n - 1)!n!}{(n - j)!(n - 1 + j)!} \\
				& & 		\quad\times \sum_{k=0}^\infty\frac{1}{k!}\left(\frac{x}{\beta(\tau)}\right)^k \frac{j_{(k)}}{(j + n)_{(k)}} \\
				& \equiv & \frac{n_{[j]}}{n_{(j)}j!} \left(\frac{x}{\beta(\tau)}\right)^j 
											\exp \left\{-\frac{x}{\beta(\tau)}\right\} M\left(j, j + n, \frac{x}{\beta(\tau)}\right),
\end{eqnarray*}
where the summation shift $l = k + j$ has been used in the second line.  The alternate integral form in Eq.~(\ref{sample:235}) follows from 
\begin{eqnarray*}
\lefteqn{e^{-x/\beta(\tau)} \sum_{k=0}^\infty\frac{1}{k!}\left(\frac{x}{\beta(\tau)}\right)^k \frac{j_{(k)}}{(j + n)_{(k)}}} \\
	& = & e^{-x/\beta(\tau)}  \sum_{k=0}^\infty\frac{1}{k!}\left(\frac{x}{\beta(\tau)}\right)^k \frac{B(j + k, n)}{B(j, n)} \\
	& = & \frac{1}{B(j, n)} e^{-x/\beta(\tau)} 
							\sum_{k=0}^\infty\frac{1}{k!}\left(\frac{x}{\beta(\tau)}\right)^k \int_0^1 v^{j + k - 1}(1 - v)^{n - 1} dv \\ 
	& = & \frac{(j + n - 1)!}{(j - 1)!(n - 1)!} \int_0^1 e^{-(1 - v)x/\beta(\tau)} v^{j + k - 1}(1 - v)^{n - 1} dv. \\
\end{eqnarray*}

The mean inter-coalescent times are
\[
\mathbb{E}\left[W^{(n)}_j\right] = \int_0^\infty\mathbb{P}\left(N_n(\tau)= j\right)d\tau
\]
which evaluates to Eq.~(\ref{sample:236}) after change of integration variable 
\[
u = \frac{1}{2\alpha\beta(\tau)} = \frac{1}{e^{\alpha\tau} - 1}, \qquad \tau = \frac{1}{\alpha} \log\left(1 + \frac{1}{u} \right), 
\]
\[
d\tau = - \frac{1}{\alpha} \frac{du}{u(u + 1)}
\]
and some further simplification.
\end{proof}
%
%%%%%%%%%%%% Subsection: Conclusions %%%%%%%%%%%%
%
\section{Summary and Conclusions}
\label{sec:Conclusions}

The results of this paper hinge on the two key theorems of Section~\ref{sec:Coalescence} which give formulae for the distribution of the number of 
ancestors at time $t - s$ of the population (Theorem~\ref{thm:PolyaAeppli}) and of a uniform random sample of size $n$ (Theorem~\ref{thm:distribAn}) from a 
Feller diffusion at a time $t$ since initiation of the process in terms of the parameter $\alpha$ of the diffusion and an initial condition $X(0) = x_0$.  
Two applications of these theorems are given.  
The first application is an efficient derivation of sample coalescent waiting times for the quasi-stationary subcritical Feller diffusion (Section~\ref{sec:CoalescenceTtoInfty}).

In the second application the coalescent tree of a supercritical Feller diffusion is identified as the limit of the RP of a BD process with appropriately 
scaled birth and death rates (Theorem~\ref{thm:coalescentTreeIsAnRP}).  The corresponding coalescent tree assuming a single ancestor is then used 
to calculate the distribution of coalescent times before the present time 
(Theorem~\ref{PoissonAncestors}) and to obtain a formula for expected coalescent times in the ancestry of a finite sample 
(Corollary~\ref{corollary:sampleCoalescence}).

With the intention of being of interest to the population genetics community, who are familiar with the \cite{kimura1964diffusion} limit 
in which the population size is taken to infinity while generation times and mutation rates shrink to zero, we have taken the approach of treating the Feller 
diffusion primarily as the continuum limit of a discrete BGW process.  The results of our theorems can be translated to all physically meaningful quantities in the 
underlying BGW process being approximated, 
including the mean $\lambda$ and variance $\sigma^2$ of the number of offspring per parent per generation, initial population counts at initiation of a BG 
process $m_0$ and $y_0$, and observed population counts at a later time $y$, via Eqs.~(\ref{BGWdiffLimit}), (\ref{alphaDef}), and (\ref{xToYScale}).  
%
%%%%%%%%%%%
%

%\pagebreak
%\section*{References}
%\bibliographystyle{elsarticle-num} 
\bibliographystyle{elsarticle-harv}\biboptions{authoryear}
%\section*{References}
%\bibliography{PopulationGenetics}

\end{document}